\documentclass[11pt]{article}

\usepackage[latin1]{inputenc}

\usepackage{amsfonts,amsmath,amssymb,amsthm,graphicx,epsfig,float,}
\usepackage[T1]{fontenc}
\usepackage[all]{xy}

\usepackage[english,francais]{babel}

{
  \newtheorem{theoreme}{Th\'eor\`eme}
  \newtheorem*{theoreme*}{Th\'eor\`eme}
  \newtheorem{lemme}[theoreme]{Lemme}
  \newtheorem{corollaire}[theoreme]{Corollaire}

  \newtheorem{proposition}[theoreme]{Proposition}
\newtheorem*{corollaire*}{Corollaire}
\newtheorem*{proposition*}{Proposition}
\theoremstyle{remark}
  \newtheorem*{remarque*}{Remarque}
}

\newcounter{ex}

\newenvironment{rem*}{
  \noindent\textbf{Remarque. }}{}

\newcommand{\Cc}{\mathbb{C}}
\newcommand{\Nn}{\mathbb{N}}

\newcommand{\Pp}{\mathbb{P}}

\title{{\bf Sur l'étude de l'entropie des applications méromorphes}}
\author{Henry de Thélin}
\date{}

\begin{document}
\maketitle


\def\figurename{{Fig.}}%
\def\proofname{Preuve}
\def\contentsname{Sommaire}%

\begin{abstract}

Nous construisons un espace adapté à l'étude de l'entropie des applications méromorphes en utilisant des limites projectives. Nous en déduisons un principe variationnel pour ces applications.

\end{abstract}

\selectlanguage{english}
\begin{center}
{\bf{ }}
\end{center}

\begin{abstract}

We construct a space which is useful in order to study the entropy of meromorphic maps by using projective limits. We deduce a variational principle for meromorphic maps.

\end{abstract}

\selectlanguage{francais}

Mots-clefs: dynamique complexe, entropie.

Classification: 32H50, 32Qxx.

\section*{{\bf Introduction}}
\par

Soit $X$ une variété complexe compacte de dimension $k$ et $f: X \longrightarrow X$ une application méromorphe dominante. C'est la donnée d'un sous-ensemble analytique $\Gamma(f)$ irréductible de dimension $k$ dans $X \times X$ (le graphe de $f$) avec ${p_1}_{| \Gamma(f)}: \Gamma(f) \longrightarrow X$ la restriction à $\Gamma(f)$ de la projection sur la première coordonnée holomorphe, surjective et dont les fibres génériques sont réduites à un point, et ${p_2}_{| \Gamma(f)}: \Gamma(f) \longrightarrow X$ la restriction à $\Gamma(f)$ de la projection sur la deuxième coordonnée holomorphe et surjective.

L'application $f$ est holomorphe en dehors d'un sous-ensemble analytique $I$ de $X$ qui est l'ensemble des points $x$ avec $\{p_1^{-1}(x) \} \cap \Gamma(f)$ de dimension supérieure ou égale à $1$. L'ensemble $I$ est de codimension au moins 2.

Un enjeu majeur en dynamique méromorphe est de calculer l'entropie topologique $h_{top}(f)$ de $f$. En effet cette question est liée à l'existence de mesures hyperboliques pour $f$ (voir \cite{Det}).

Pour certaines variétés complexes compactes, on a l'existence d'outils qui permettent d'étudier cette entropie: ce sont les degrés dynamiques. Voici leurs constructions: pour $(X, \omega)$ variété kählérienne compacte et $l=0, \cdots, k$, si on note $[\Gamma(f)]$ le courant d'intégration sur $\Gamma(f)$, la forme $f^* (\omega^l)=(p_1)_{*}(p_2^* (\omega^l) \wedge [\Gamma(f)])$ est à coefficients $L^1$. On peut donc considérer 

$$\delta_l(f)= \int f^*(\omega^l) \wedge \omega^{k-l},$$

et le $l$-ème degré dynamique est défini par $d_l= \lim_{n \rightarrow + \infty} (\delta_l(f^n))^{1/n}$. L'existence de cette limite a été obtenue par A. Russakovskii et B. Shiffman dans le cas où $X=\Pp^k(\Cc)$ (voir \cite{RS}) et par T.-C. Dinh et N. Sibony quand $(X, \omega)$ est une variété kählérienne compacte (voir \cite{DS1} et \cite{DS2}).

Dans ce contexte, M. Gromov (voir \cite{Gr}) pour le cas où $f$ est holomorphe et T.-C. Dinh et N. Sibony (voir \cite{DS1} et \cite{DS2}) pour le cas méromorphe ont montré que l'entropie topologique de $f$ est toujours majorée par $\max_{l=0, \cdots, k} \log d_l$. 

Il y a de nombreux travaux qui portent sur la minoration de cette entropie topologique ainsi que sur la construction de mesures avec une entropie métrique maximale. La situation est beaucoup plus délicate que pour les applications lisses: par exemple, il existe des applications méromorphes de $\Pp^2(\Cc)$ avec $d_1=d_2=2$ (le degré topologique vaut donc $2$) qui sont d'entropie nulle (voir \cite{Gu}). Dans ce type d'exemple la dynamique se concentre en particulier sur l'ensemble d'indétermination. Dans \cite{Gu}, V. Guedj conjecture que lorsqu'il y a un degré dynamique qui domine strictement tous les autres alors $h_{top}(f)=\max_{l=0, \cdots, k} \log d_l$.

L'idée dans cet article va être de considérer des éclatements de $X$ dans le lieu d'indétermination de $f$ et de relever $f$ en une application méromorphe dans cet espace. Ensuite nous recommencerons avec la nouvelle application et nous produirons donc une suite d'espaces $X_n$ et des applications méromorphes $F_n: X_n \longrightarrow X_n$ qui relèvent $f$. Enfin, il s'agira de prendre une limite projective sur les $(X_n)$. Nous obtiendrons ainsi un espace particulièrement bien adapté à l'étude de la dynamique de $f$. Signalons qu'une telle construction a été réalisée dans \cite{HP} pour un cas particulier et notons aussi le lien avec \cite{SFJ} et \cite{Ca} où les auteurs considèrent l'espace constitué de tous les éclatements de la variété.

Nous détaillerons le procédé d'éclatements dans le paragraphe \ref{éclatements}. Nous aboutirons à une situation qui peut être formalisée de la façon plus générale suivante.

Notons $X_0=X$ et $F_0=f$ et admettons que l'on ait une suite $(X_n)$ de variétés complexes compactes de dimension $k$ et $F_n: X_n \longrightarrow X_n$ une suite d'applications méromorphes dominantes telles que pour tout $n \geq 1$ on ait le diagramme commutatif suivant:

$$
\xymatrix{
X_n  \ar@{.>}[r]^{F_n} \ar[d]_{\pi_n} \ar[rd]^{s_n} & X_n \ar[d]^{\pi_n}\\
X_{n-1} \ar@{.>}[r]^{F_{n-1}} & X_{n-1}
}
$$

où $\pi_n : X_n \longrightarrow X_{n-1}$ est holomorphe, surjective et avec ses fibres génériques réduites à un point et $s_n : X_n \longrightarrow X_{n-1}$ holomorphe et surjective. Nous supposerons aussi que
$$ \{ x \in X_{n-1} \mbox{  ,  } \mbox{dim}({\pi_n^{-1}(x)}) \geq 1 \} \subset I(F_{n-1}) $$
où $I(F_{n-1})$ est l'ensemble d'indétermination de $F_{n-1}$. 

On considère alors la limite projective $X_{\infty} = \varprojlim X_n$ qui est simplement ici
$$X_{\infty} = \{ (x_n) \in \prod_{n \geq 0} X_n \mbox{  ,  } \pi_n(x_n)=x_{n-1} \mbox{  pour tout  } n \geq 1 \}.$$

Nous munissons $X_n$ d'une métrique $dist_n$ compatible avec sa topologie (pour tout $n \geq 0$) et sur $X_{\infty}$ on définit la métrique

$$\delta(\widehat{x}, \widehat{y})= \sum_{n=0}^{+ \infty} \frac{dist_n(x_n,y_n)}{\alpha_n diam(X_n)}$$

avec $\widehat{x}=(x_n)$ et $\widehat{y}=(y_n)$ dans $X_{\infty}$, $diam(X_n)$ le diamètre de $X_n$ pour $dist_n$ et $(\alpha_n)$ une suite de réels strictement positifs avec $\sum_{n \geq 0} \frac{1}{\alpha_n} < + \infty$. Cette métrique rend $X_{\infty}$ compact.

Nous allons maintenant définir un opérateur $\sigma: X_{\infty} \longrightarrow X_{\infty}$ qui va relever $f$. 

Pour $\widehat{x}=(x_n) \in X_{\infty}$, on pose

$$\sigma(\widehat{x})=(s_n(x_n))_{n \geq 1} \in X_{\infty}.$$

L'opérateur $\sigma$ résout l'indétermination de $f$. En effet, on a

\begin{proposition}{\label{prop1}}
L'application $\sigma$ est continue.
\end{proposition}

Nous avons plusieurs systèmes dynamiques: les $F_n: X_n \longrightarrow X_n$ (dont le système initial $f: X \longrightarrow X$) et $\sigma: X_{\infty} \longrightarrow X_{\infty}$. Le premier objectif de cet article est de donner le lien entre ces quantités:

\begin{theoreme}{\label{théorème1}}
La suite $(h_{top}(F_n))_n$ est croissante et

$$h_{top}(\sigma)= \sup_{n \geq 0} h_{top}(F_n).$$

\end{theoreme}

Pour l'exemple de V. Guedj énoncé plus haut, nous verrons que $h_{top}(\sigma)= \log 2 = \max \log d_l$. Cela renforce l'idée que $\sigma$ résout l'indétermination de $f$.

Remarquons que l'espace $X_{\infty}$ et l'application $\sigma$ dépendent de la suite $(X_n)$ considérée. Nous donnerons au paragraphe \ref{éclatements} un exemple de construction mais il peut très bien y en avoir d'autres. Par ailleurs, nous verrons que les quantités ci-dessus ne dépendent pas du choix des distances $dist_n$ ou de la suite $(\alpha_n)$.

Le second objectif de cet article est de donner une application de la construction de l'espace $X_{\infty}$.

Pour $A \subset X$, on pose $f^{-1}(A)=p_1(p_2^{-1}(A) \cap \Gamma(f))$ et

$$\widetilde{f^{-n}(A)}= f^{-1} ( \cdots ( f^{-1}(A)) \cdots )$$

où le $f^{-1}$ apparaît $n$ fois (pour $n \in \Nn$). Remarquons que $\widetilde{f^{-n}(A)}$ peut-être différent de $f^{-n}(A)$, en particulier quand $f$ n'est pas algébriquement stable.

Le deuxième théorème de cet article est le principe variationnel suivant:

\begin{theoreme}{\label{théorème2}}
 
On suppose tous les $\widetilde{f^{-n}(I)}$ disjoints pour $n \in \Nn$. Alors

$$h_{top}(\sigma)= h_{top}(f)= \sup \{h_{\nu}(f) \mbox{  ,  } \nu \mbox{ ergodique et } \nu (I)=0 \}.$$

\end{theoreme}

Remarquons que le principe variationnel n'est pas vrai en toute généralité (voir l'exemple 3.1 de \cite{Gu}), il est donc nécessaire de mettre une hypothèse sur l'ensemble d'indétermination.

Notons aussi que ce théorème est plus fort que le principe variationnel classique, qui est

$$h_{top}(f)= \sup \{h_{\nu}(f) \mbox{  ,  } \nu \mbox{ ergodique et } \nu(I)=0 \}.$$

Voici le plan de cet article: dans un premier paragraphe, nous donnons une construction possible des espaces $X_n$ et des diagrammes précédents. Dans le second, nous prouvons la proposition \ref{prop1} et le théorème \ref{théorème1}. Ensuite nous détaillerons le calcul de $h_{top}(\sigma)$ sur un exemple: celui de V. Guedj où $d_1=d_2=2$ et $h_{top}(f)=0$. Nous verrons que dans ce cas $h_{top}(\sigma)= \log 2$. Enfin, dans le dernier paragraphe, nous démontrerons le principe variationnel.

{\bf Remerciement:} je remercie Tien-Cuong Dinh pour les discussions que nous avons eues au sujet de cet article.

\section{\bf Construction des suites d'éclatements}{\label{éclatements}}

Dans ce paragraphe nous produisons une suite de variétés complexes compactes $(X_n)$ et d'applications méromorphes $F_n: X_n \longrightarrow X_n$ qui vérifient le formalisme général donné dans l'introduction.

On part de la variété complexe compacte $X_0=X$ de dimension $k$, munie d'une distance $dist_0$ et de $F_0=f$ méromorphe dominante de $X$ dans $X$.

Les singularités de $\Gamma(f)$ sont dans $p_1^{-1}(I) \cap \Gamma(f)$. D'après le théorème de désingularisation d'Hironaka (voir \cite{Hi}), il existe une variété complexe compacte $\widetilde{X \times X}$, une sous-variété $\widetilde{\Gamma(f)}$ de $\widetilde{X \times X}$ et une application holomorphe $\pi: \widetilde{X \times X} \longrightarrow X \times X$ (qui est une composée d'éclatements) telles que $\pi$ soit un biholomorphisme de $ \widetilde{\Gamma(f)} \setminus (\pi^{-1}( p_1^{-1}(I) \cap \Gamma(f)))$ dans $\Gamma(f) \setminus (p_1^{-1}(I) \cap \Gamma(f))$. On a $\pi(\widetilde{\Gamma(f)})= \Gamma(f)$.

Remarquons que lorsque $X$ est kählérienne, $X \times X$ l'est aussi. En particulier, sous cette hypothèse, $\widetilde{X \times X} $ est une variété kählérienne par un théorème de Blanchard (voir \cite{Bl}) et $\widetilde{\Gamma(f)}$ aussi en tant que sous-variété de $\widetilde{X \times X} $.

Notons $X_1= \widetilde{\Gamma(f)}$, $\pi_1=p_1 \circ \pi$ et $s_1= p_2 \circ \pi$. On a $\pi_1$ holomorphe, surjective avec ses fibres génériques réduites à un point et $s_1$ holomorphe et surjective. Ainsi on obtient le diagramme

$$
\xymatrix{
X_1   \ar[d]_{\pi_1} \ar[rd]^{s_1} \\
X_{0} \ar@{.>}[r]^{F_{0}} & X_{0}
}
$$

L'application $\pi_1$ est biméromorphe. On peut donc considérer $F_1= \pi_1^{-1} \circ F_0 \circ \pi_1$ qui est bien définie en dehors d'un sous-ensemble analytique de $X_1$. En prenant l'adhérence du graphe de $F_1$ dans $X_1 \times X_1$ on obtient ainsi une nouvelle application méromorphe $F_1: X_1 \longrightarrow X_1$. Comme $F_0$ est dominante, $F_1$ l'est aussi. On a bien obtenu le diagramme voulu dans l'introduction avec $n=1$. Remarquons aussi que par construction

$$ \{ x \in X_{0} \mbox{  ,  } \mbox{dim}({\pi_1^{-1}(x)}) \geq 1 \} \subset I(F_{0})=I.$$

On peut maintenant recommencer tout ce que l'on vient de faire avec $X_1$ à la place de $X_0$ et $F_1$ au lieu de $F_0$. En itérant le procédé, on obtient ainsi une suite de variétés complexes compactes $(X_n)$ de dimension $k$ et des applications méromorphes $F_n: X_n \longrightarrow X_n$ qui vérifient les conditions demandées dans l'introduction.

Remarquons que lorsque $X=X_0$ est kählérienne, tous les $X_n$ le sont aussi. Comme T.-C. Dinh et N. Sibony ont montré que les degrés dynamiques sont des invariants biméromorphes (voir \cite{DS1} p. 961), on en déduit par récurrence que

\begin{lemme}

Lorsque $X$ est kählérienne, on a

$$d_l(F_n)=d_l(F_0)=d_l(f)$$

pour tout $l=0, \cdots , k$ et $n \geq 0$.

\end{lemme}

\section{Démonstration de la proposition \ref{prop1} et du théorème \ref{théorème1}}

Dans ce paragraphe, nous allons tout d'abord montrer que les résultats ne dépendent pas des distances $dist_n$ sur $X_n$ et de la suite $(\alpha_n)$ choisies. Ensuite, nous donnerons une construction de distances qui permettront de simplifier les preuves. Nous montrerons alors la continuité de $\sigma$ (proposition \ref{prop1}), puis que la suite $(h_{top}(F_n))_n$ est croissante. Enfin, nous prouverons l'égalité

$$h_{top}(\sigma)= \sup_{n \geq 0} h_{top}(F_n).$$

\subsection{\bf Indépendance des résultats aux métriques}{\label{indépendance}}

Tout d'abord, comme $X_n$ est compacte, l'entropie topologique $h_{top}(F_n)$ ne dépend pas de la métrique $dist_n$ (qui définit sa topologie) que l'on a choisie (voir la proposition 3.1.2 de \cite{KH}). Cela provient du fait que deux métriques topologiquement équivalentes sont uniformément équivalentes.

Considérons maintenant deux métriques sur $X_{\infty}$ construites comme dans l'introduction:

$$\delta(\widehat{x}, \widehat{y})= \sum_{n=0}^{+ \infty} \frac{dist_n(x_n,y_n)}{\alpha_n diam(X_n)} \mbox{   et   }  \delta'(\widehat{x}, \widehat{y})= \sum_{n=0}^{+ \infty} \frac{dist_n'(x_n,y_n)}{\beta_n diam'(X_n)} $$

avec $\widehat{x}=(x_n)$ et $\widehat{y}=(y_n)$ dans $X_{\infty}$, $dist_n$ et $dist_n'$ des métriques sur $X_n$ qui définissent sa topologie, $diam(X_n)$ le diamètre de $X_n$ pour $dist_n$, $diam'(X_n)$ celui pour $dist'_n$, $(\alpha_n)$ et $(\beta_n)$ des suites de réels strictement positifs avec $\alpha= \sum_{n \geq 0} \frac{1}{\alpha_n} < + \infty$ et $\beta= \sum_{n \geq 0} \frac{1}{\beta_n} < + \infty$.

Montrons qu'elles sont topologiquement équivalentes. Comme $X_{\infty}$ est compact, on aura comme précédemment que l'entropie topologique de $\sigma$ pour $\delta$ sera la même que l'entropie topologique de $\sigma$ pour $\delta'$.

Soit $\epsilon >0$. On choisit $n_0 \in \Nn$ tel que $\sum_{n \geq n_0} \frac{1}{\alpha_n} < \frac{\epsilon}{2}$. 

Fixons $n$ compris entre $0$ et $n_0 -1$. L'application $Id:(X_n, dist_n') \longrightarrow (X_n, dist_n)$ est uniformément continue. Il existe donc $\eta_n > 0$ tel que pour tout $x,y \in X_n$ avec $dist_n'(x,y) < \eta_n$ on ait $dist_n(x,y) < \frac{\epsilon}{2\alpha} diam(X_n)$.

Soit

$$\eta= \min_{n=0, \cdots, n_0 -1} \left( \frac{\eta_n}{\beta_n diam'(X_n)} \right).$$

Pour $\widehat{x}=(x_n)$ et $\widehat{y}=(y_n)$ dans $X_{\infty}$ avec 

$$\delta'(\widehat{x}, \widehat{y})=\sum_{n=0}^{+ \infty} \frac{dist_n'(x_n ,y_n)}{\beta_n diam'(X_n)} < \eta$$

on a $dist'(x_n,y_n)< \eta_n$ pour $n=0, \cdots, n_0 -1$. D'où

\begin{equation*}
\begin{split}
\delta(\widehat{x},\widehat{y}) &= \sum_{n=0}^{+ \infty} \frac{dist_n(x_n ,y_n)}{\alpha_n diam(X_n)} \leq \sum_{n=0}^{n_0 -1} \frac{ dist_n(x_n ,y_n)}{\alpha_n diam(X_n)} + \frac{\epsilon}{2}\\
& \leq \sum_{n=0}^{n_0 -1} \frac{ \epsilon diam(X_n)}{2 \alpha \alpha_n diam(X_n)} + \frac{\epsilon}{2} \leq \epsilon.
\end{split}
\end{equation*}

L'application $Id: (X_{\infty} , \delta') \longrightarrow (X_{\infty} , \delta)$ est donc uniformément continue. Par le même raisonnement, on montre que $Id: (X_{\infty} , \delta) \longrightarrow (X_{\infty} , \delta')$ l'est aussi, d'où l'équivalence des métriques $\delta$ et $\delta'$.

L'entropie topologique $h_{top}(\sigma)$ et la continuité de $\sigma$ sont donc indépendantes du choix des métriques $dist_n$ sur $X_n$ et de la suite $(\alpha_n)$.

$$ $$

Nous allons maintenant construire une suite de distances $dist_n$ sur $X_n$ qui aura une propriété de croissance qui permettra de simplifier les démonstrations.

Pour cela, on part de métriques $dist_n'$ (pour $n \in \Nn$) qui définissent la topologie de $X_n$. En particulier, les applications $\pi_n: X_n \longrightarrow X_{n-1}$ sont holomorphes donc continues. Notons $dist_0=dist_0'$ et $dist_1(x,y)=dist_1'(x,y)+ dist_0(\pi_1(x),\pi_1(y))$ (pour $x,y \in X_1$). On a $dist_1$ et $dist_1'$ topologiquement équivalentes car $\pi_1$ est continue. En recommençant le procédé, on obtient une suite de distances $dist_n$ qui vérifient que pour tout $n \geq 1$ et tout $x,y \in X_n$, on a $dist_n(x,y) \geq dist_{n-1}(\pi_1(x),\pi_1(y))$.

Comme on a vu que les résultats ne dépendent pas du choix des distances $dist_n$ sur $X_n$ ainsi que de la suite $(\alpha_n)$, dans toute la suite de l'article, nous prendrons la suite de distances $dist_n$ que nous venons de construire et $(\alpha_n)=(2^n)$.

Ainsi

$$\delta(\widehat{x}, \widehat{y})= \sum_{n=0}^{+ \infty} \frac{dist_n(x_n,y_n)}{2^n diam(X_n)}.$$

\subsection{\bf Continuité de $\sigma$}

Fixons $\widehat{x}=(x_n) \in X_{\infty}$ et $\epsilon > 0$. Soit $(\widehat{x_m})$ une suite de $ X_{\infty}$ qui converge vers $\widehat{x}$.

On écrit $\widehat{x_m}= (x_{m,n})$ avec $x_{m,n} \in X_n$ pour tout $n \geq 0$. En particulier, par définition de $\sigma$, on a $\sigma(\widehat{x_m})=(s_{n}(x_{m,n}))_{n \geq 1}$ et  $\sigma(\widehat{x})=(s_{n}(x_{n}))_{n \geq 1}$.

Pour $n_0$ assez grand on a $\sum_{n \geq n_0} \frac{1}{2^n} < \frac{\epsilon}{2}$ et alors

\begin{equation*}
\begin{split}
\delta(\sigma(\widehat{x_m}), \sigma(\widehat{x}))&= \sum_{n=0}^{n_0 - 1} \frac{dist_n(s_{n+1}(x_{m,n+1}),s_{n+1}(x_{n+1}))}{2^n diam(X_n)}\\
& + \sum_{n=n_0}^{+ \infty} \frac{dist_n(s_{n+1}(x_{m,n+1}),s_{n+1}(x_{n+1}))}{2^n diam(X_n)}\\
& \leq \sum_{n=0}^{n_0 - 1} \frac{dist_n(s_{n+1}(x_{m,n+1}),s_{n+1}(x_{n+1}))}{2^n diam(X_n)} + \frac{\epsilon}{2}.\\
\end{split}
\end{equation*}

Comme on a $(\widehat{x_m})$ qui converge vers $\widehat{x}$ quand $m \rightarrow + \infty$, la suite $(x_{m,n})_m$ converge vers $x_n$ quand $m \rightarrow + \infty$ pour $dist_n$ (pour tout $n \geq 0$). Les applications $s_n$ sont holomorphes donc $(s_n(x_{m,n}))_m$ converge vers $s_n(x_n)$ quand $m$ tend vers l'infini pour la métrique $dist_{n-1}$ pour tout $n \geq 1$.

En particulier, $\sum_{n=0}^{n_0 - 1} \frac{dist_n(s_{n+1}(x_{m,n+1}),s_{n+1}(x_{n+1}))}{2^n diam(X_n)}$ converge vers $0$ quand $m \rightarrow + \infty$ car il n'y a qu'un nombre fini de termes dans la somme: cette quantité est donc plus petite que $\frac{\epsilon}{2}$ pour $m$ assez grand et la proposition \ref{prop1} est démontrée.

\subsection{\bf Croissance  de $(h_{top}(F_n))_n$}{\label{croissance}}

Rappelons tout d'abord la définition de $(h_{top}(F_n))_n$ (voir par exemple \cite{Gu}).

Notons pour cela 

$$\Omega_n=X_n \setminus \cup_{m \in \Nn} F_n^{-m}(I(F_n))$$

où $I(F_n)$ est l'ensemble d'indétermination de $F_n$. L'ensemble $\Omega_n$ est dense dans $X_n$ car $F_n$ est dominante et il est invariant par $F_n$. Alors, on a

$$h_{top}(F_n)= \lim_{\epsilon \rightarrow 0} \limsup_{m \rightarrow + \infty} \frac{1}{m} \log \max ( \# G \mbox{  ,  } G  \mbox{  ensemble  } (m, \epsilon) \mbox{-séparé dans  } \Omega_n \mbox{  pour  } F_n ).$$

Montrons que la suite $(h_{top}(F_n))_n$ est croissante. On fixe $n \geq 1$.

Soit $\gamma > 0$. Pour $\epsilon$ assez petit, on a 

$$ \limsup_{m \rightarrow + \infty} \frac{1}{m} \log \max ( \# G \mbox{, } G  \mbox{ ens. } (m, \epsilon) \mbox{-séparé dans } \Omega_{n - 1} \mbox{ pour } F_{n-1} ) \geq h_{top}(F_{n-1}) - \gamma.$$

Soit $m_0 \in \Nn$. On peut trouver $m \geq m_0$ et $\{x_1, \cdots, x_N \}$ un ensemble $(m, \epsilon)$-séparé dans $\Omega_{n-1}$ pour $F_{n-1}$ avec $N \geq e^{(h_{top}(F_{n-1}) - 2 \gamma)m}$.

Si on fixe une forme volume sur $X_{n-1}$, on voit que l'ensemble $\pi_n(\cup_{m \in \Nn} F_n^{-m}(I(F_n)))$ est de volume nul car $F_n$ est dominante et $\pi_n$ holomorphe. Il en est de même du complémentaire de $\Omega_{n-1}$. En particulier, quitte à bouger un peu les $x_i$, on peut produire des points $x_1' , \cdots , x_N'$ qui sont $(m, \frac{\epsilon}{2})$-séparés pour $F_{n-1}$ et dans $\Omega_{n-1} \setminus \pi_n(\cup_{m \in \Nn} F_n^{-m}(I(F_n)))$. En effet, comme les points $x_i$ sont dans $\Omega_{n-1}$, les itérées de $F_{n-1}$ sont continues en ces points là.

L'application $\pi_n$ est surjective, on peut donc trouver $y_1 , \cdots , y_N$ dans $X_n$ avec $\pi_n(y_i)=x_i'$ pour $i=1 , \cdots , N$. Par construction les $y_i$ sont dans $\Omega_n$.

Montrons maintenant que les points $y_i$ sont $(m, \frac{\epsilon}{2})$-séparés pour $F_{n}$. 

Si $i \neq j$ avec $1 \leq i,j \leq N$, les points $x_i'$ et $x_j'$ sont $(m, \frac{\epsilon}{2})$-séparés pour $F_{n-1}$. Il existe donc $0 \leq l \leq m-1$ avec $dist_{n-1}(F_{n-1}^l(x_i'),F_{n-1}^l(x_j') ) \geq \frac{\epsilon}{2}$. Par l'hypothèse faite sur les distances, on a

$$dist_{n}(F_{n}^l(y_i),F_{n}^l(y_j) ) \geq dist_{n-1}(\pi_n(F_{n}^l(y_i)),\pi_n(F_{n}^l(y_j))).$$

Par définition, on a $\pi_n \circ F_n^l= F_{n-1}^l \circ \pi_n$ en dehors d'un sous-ensemble analytique de $X_n$. Comme $y_i$ est dans $\Omega_n$, on a $F_n^l$ continue en $y_i$ et puisque $x_i'=\pi_n(y_i)$ est dans $\Omega_{n-1}$, on a $F_{n-1}^l$ continue en $\pi_n(y_i)$. On a donc $\pi_n(F_{n}^l(y_i))= F_{n-1}^l( \pi_n(y_i))$. Il en est de même pour $y_j$. En combinant cela avec l'inégalité ci-dessus, on obtient 

$$dist_{n}(F_{n}^l(y_i),F_{n}^l(y_j) ) \geq dist_{n-1}(F_{n-1}^l(x_i'),F_{n-1}^l(x_j') ) \geq \frac{\epsilon}{2}.$$

On a donc montré que pour tout $m_0$, il existe $m \geq m_0$ et un ensemble $G$ de $\Omega_n$ qui est $(m, \frac{\epsilon}{2})$-séparé pour $F_{n}$ avec $\# G \geq e^{(h_{top}(F_{n-1}) - 2 \gamma)m}$.

En particulier,

$$ \limsup_{m \rightarrow + \infty} \frac{1}{m} \log \max ( \# G \mbox{, } G  \mbox{ ensemble } (m, \frac{\epsilon}{2}) \mbox{-séparé dans } \Omega_{n } \mbox{ pour } F_{n} ) \geq h_{top}(F_{n-1}) -2 \gamma.$$

En faisant tendre $\epsilon$ puis $\gamma$ vers $0$, on obtient $h_{top}(F_{n}) \geq h_{top}(F_{n-1})$. C'est ce que l'on voulait démontrer.

\subsection{\bf Démonstration de l'égalité $h_{top}(\sigma)= \sup_{n \geq 0} h_{top}(F_n)$}

Nous commençons par montrer que pour tout $n \geq 0$ on a $h_{top}(\sigma) \geq h_{top}(F_n)$ puis que $h_{top}(\sigma)= \sup_{n \geq 0} h_{top}(F_n)$.

\subsubsection{\bf Preuve de $h_{top}(\sigma) \geq h_{top}(F_n)$}

Soit $\gamma > 0$. Pour $\epsilon$ assez petit, on a 

$$ \limsup_{m \rightarrow + \infty} \frac{1}{m} \log \max ( \# G \mbox{, } G  \mbox{ ensemble } (m, \epsilon) \mbox{-séparé dans } \Omega_{n} \mbox{ pour } F_{n} ) \geq h_{top}(F_{n}) - \gamma.$$

Soit $m_0 \in \Nn$. On peut trouver $m \geq m_0$ et $\{x_1, \cdots, x_N \}$ un ensemble $(m, \epsilon)$-séparé dans $\Omega_{n}$ pour $F_{n}$ avec $N \geq e^{(h_{top}(F_{n}) - 2 \gamma)m}$.

Les applications $\pi_l$ sont surjectives, il existe donc $\widehat{x_i} \in X_{\infty}$ avec 

$$\widehat{x_i}= ( \cdots, x_i, \pi_n(x_i), \cdots , \pi_1( \cdots (\pi_n(x_i))))$$

pour $i=1, \cdots , N$.

Montrons que les points $\widehat{x_i}$ sont $\left( m, \frac{\epsilon}{2^n diam(X_n)} \right)$-séparés dans $X_{\infty}$ pour $\sigma$.

Soit $1 \leq i,j \leq N$ avec $i \neq j$. On a l'existence de $0 \leq l \leq m-1$ avec 

$$dist_n (F^l_n(x_i), F^l_n(x_j)) \geq \epsilon.$$

Maintenant,

$$\delta( \sigma^l(\widehat{x_i}), \sigma^l(\widehat{x_j}))= \sum_{p=0}^{+ \infty} \frac{ dist_p((\sigma^l(\widehat{x_i}))_p , (\sigma^l(\widehat{x_j}))_p)}{2^p diam(X_p)}$$

où $\sigma^l(\widehat{x_i})= ( \cdots , (\sigma^l(\widehat{x_i}))_p , \cdots , (\sigma^l(\widehat{x_i}))_0)$.

Calculons $(\sigma^l(\widehat{x_i}))_p$.

Notons $\widehat{x_i}=(\cdots , x_{i,p}, \cdots , x_{i,0})$ (avec $x_{i,n}=x_i$). On a 

$$\sigma(\widehat{x_i})=(\cdots , s_{p+1}(x_{i,p+1}), \cdots , s_1(x_{i,1}))$$

et en recommençant $l$ fois

$$\sigma^l(\widehat{x_i})=(\cdots , s_{p+1}\circ \cdots \circ s_{p+l}(x_{i,p+l}), \cdots , s_1 \circ \cdots \circ s_l(x_{i,l})).$$

Ainsi $(\sigma^l(\widehat{x_i}))_p =s_{p+1}\circ \cdots \circ s_{p+l}(x_{i,p+l}) $ pour $p \geq 0$. Mais

\begin{lemme}{\label{lemme4}}

Pour tout $l \geq 1$ et $p\geq 0$, on a $F_p^l \circ \pi_{p+1} \circ \cdots \circ \pi_{p+l}= s_{p+1} \circ \cdots \circ s_{p+l}$.

\end{lemme}

\begin{proof}

Rappelons que lorsque $h: X \longrightarrow Y$ et $g: Y \longrightarrow Z$ sont des applications méromorphes dominantes entre variétés complexes compactes, la composée $g \circ h$ est définie par son graphe dans $X \times Z$ obtenu en prenant l'adhérence dans $X \times Z$ de l'ensemble des points $\{(x, g(h(x))) \mbox{  ,  } x \notin I(h) \mbox{  ,  } h(x) \notin I(g) \}$. Par ailleurs deux applications méromorphes dominantes sont égales si elles ont même graphe, ou ce qui revient au même, si elles coïncident sur un ouvert où elles sont toutes les deux holomorphes.

Fixons $p \geq 0$ et faisons une récurrence sur $l \geq 1$.

Pour $l=1$ on a $F_p \circ \pi_{p+1}= s_{p+1}$ grâce au diagramme que l'on a supposé dans l'introduction.

Supposons la propriété vraie au rang $l$. L'ensemble

$$\mathcal{E}=\{x \in X_{p+l} \mbox{  ,  } F_{p}^q \circ \pi_{p+1} \circ \cdots \circ \pi_{p+l}(x) \notin I(F_p) \mbox{  pour  } q=0, \cdots, l\}$$

est le complémentaire d'un sous-ensemble analytique de $X_{p+l}$ car $F_p$ et les $\pi_m$ sont dominantes.

Si $x \in \pi_{p+l+1}^{-1}(\mathcal{E})$, on a

$$F_p^{l+1} \circ \pi_{p+1} \circ \cdots \circ \pi_{p+l+1}(x)= F_p \circ s_{p+1} \circ \cdots \circ s_{p+l} \circ \pi_{p+l+1}(x)$$

par hypothèse de récurrence.

Maintenant soit

$$\mathcal{F}=\{ x \in X_{p+l+1} \mbox{  ,  } s_{p+q} \circ \cdots s_{p+l} \circ \pi_{p+l+1}(x) \notin I(F_{p+q-1}) \mbox{  pour  } q=1, \cdots, l+1\}.$$

$\mathcal{F}$ est le complémentaire d'un sous-ensemble analytique de $X_{p+l+1}$ et par le diagramme, si $x \in \mathcal{F}$ on a

\begin{equation*}
\begin{split}
F_p \circ s_{p+1} \circ \cdots \circ s_{p+l} \circ \pi_{p+l+1}(x)&=s_{p+1} \circ F_{p+1} \circ s_{p+2} \circ \cdots \circ s_{p+l} \circ \pi_{p+l+1}(x) = \cdots\\
&= s_{p+1} \circ \cdots \circ s_{p+l} \circ F_{p+l} \circ \pi_{p+l+1}(x)\\
&= s_{p+1} \circ \cdots \circ s_{p+l} \circ s_{p+l+1}(x).
\end{split}
\end{equation*}

Pour $x \in  \pi_{p+l+1}^{-1}(\mathcal{E}) \cap \mathcal{F}$ on a bien 

$$F_p^{l+1} \circ \pi_{p+1} \circ \cdots \circ \pi_{p+l+1}(x)=s_{p+1} \circ \cdots \circ s_{p+l} \circ s_{p+l+1}(x)$$

ce qui démontre le lemme.

\end{proof}

Comme le point $\pi_{n+1} \circ \cdots \circ \pi_{n+l} (x_{i,n+l})= x_{i,n}=x_i \in \Omega_n$ (là où les $F_n^l$ sont holomorphes), en appliquant le lemme précédent pour $p=n$, on obtient

$$(\sigma^l(\widehat{x_i}))_n= F_n^l \circ \pi_{n+1} \circ \cdots \circ \pi_{n+l} (x_{i,n+l})=  F_n^l(x_i).$$

De même $(\sigma^l(\widehat{x_j}))_n = F_n^l(x_j)$. Finalement,

\begin{equation*}
\begin{split}
\delta( \sigma^l(\widehat{x_i}), \sigma^l(\widehat{x_j}))&= \sum_{p=0}^{+ \infty} \frac{ dist_p((\sigma^l(\widehat{x_i}))_p , (\sigma^l(\widehat{x_j}))_p)}{2^p diam(X_p)}\\
& \geq \frac{ dist_n((\sigma^l(\widehat{x_i}))_n , (\sigma^l(\widehat{x_j}))_n)}{2^n diam(X_n)}\\
& = \frac{ dist_n(F_n^l(x_i) , F_n^l(x_j))}{2^n diam(X_n)} \geq \frac{\epsilon}{2^n diam(X_n)}.
\end{split}
\end{equation*}

Les points $\widehat{x_i}$ sont donc bien $\left( m, \frac{\epsilon}{2^n diam(X_n)} \right)$-séparés.

Ainsi,

\begin{equation*}
\begin{split}
&\limsup_{m \rightarrow + \infty} \frac{1}{m} \log \max ( \# G \mbox{, } G  \mbox{ ensemble } ( m, \frac{\epsilon}{2^n diam(X_n)} ) \mbox{-séparé dans } X_{\infty} \mbox{ pour } \sigma )\\
 &\geq h_{top}(F_{n}) - 2 \gamma.
\end{split}
\end{equation*}

En faisant tendre $\epsilon$ puis $\gamma$ vers $0$, on obtient

$$h_{top}(\sigma) \geq h_{top}(F_{n})$$

pour tout $n \geq 0$.

\subsubsection{\bf Fin de la preuve de $h_{top}(\sigma)= \sup_{n \geq 0} h_{top}(F_n)$}

Soit $\gamma > 0$. Pour $\epsilon$ assez petit, on a 

$$ \limsup_{m \rightarrow + \infty} \frac{1}{m} \log \max ( \# G \mbox{, } G  \mbox{ ensemble } (m, \epsilon) \mbox{-séparé dans } X_{\infty} \mbox{ pour } \sigma ) \geq h_{top}(\sigma) - \gamma.$$

Soit $n_0 \in \Nn$ tel que $\sum_{n \geq n_0} \frac{1}{2^n} < \frac{\epsilon}{4}$. On fixe $m \geq 1$ et on considère $\widehat{x_1}, \cdots , \widehat{x_N}$ un ensemble maximal $(m, \epsilon)$-séparé dans $X_{\infty}$ pour l'application $\sigma$ et la métrique $\delta$.

Notons

\begin{equation*}
\begin{split}
\mathcal{I}=& \cup_{q \geq 0} F_{m+n_0}^{-q}(I(F_{m+n_0})) \cup \pi_{m+n_0}^{-1}(\cup_{q \geq 0} F_{m+n_0-1}^{-q}(I(F_{m+n_0-1}))) \cup \\
&\cdots \cup (\pi_1 \circ \cdots \circ \pi_{m+n_0})^{-1}(\cup_{q \geq 0} F_{0}^{-q}(I(F_{0}))).
\end{split}
\end{equation*}

Chaque $\widehat{x_i}$ s'écrit $\widehat{x_i}=( \cdots , x_{i,p} , \cdots , x_{i,0})$.

Si on met une forme volume sur $X_{m+n_0}$, on a que $\mathcal{I}$ est de mesure nulle (car les $F_p$ et les $\pi_p$ sont dominantes). En particulier, on peut trouver $x_{i, m+n_0}' \in X_{m+n_0} \setminus \mathcal{I}$ suffisamment proche de $x_{i, m+n_0}$ pour que

$$dist_{q-1}(s_q \circ \cdots \circ s_p \circ \pi_{p+1} \circ \cdots \circ \pi_{m+n_0}(x_{i, m+n_0}), s_q \circ \cdots \circ s_p \circ \pi_{p+1} \circ \cdots \circ \pi_{m+n_0}(x_{i, m+n_0}')) < \frac{\epsilon}{8}$$

pour tout $p=0, \cdots, m+n_0$ et $q=1, \cdots , p+1$. En effet toutes les applications $s_q \circ \cdots \circ s_p \circ \pi_{p+1} \circ \cdots \circ \pi_{m+n_0}$ sont holomorphes donc continues.

Comme les $\pi_q$ sont surjectives on peut compléter $x_{i, m+n_0}'$ pour obtenir un point $\widehat{x_i}' \in X_{\infty}$

$$\widehat{x_i}'=( \cdots, x_{i, m+n_0}' , \pi_{m+n_0}(x_{i, m+n_0}'), \cdots , \pi_1 \circ \cdots \circ \pi_{m+n_0}(x_{i, m+n_0}')).$$

Par construction, si on écrit $\widehat{x_i}'=( \cdots , x_{i,p}' , \cdots , x_{i,0}')$, on a 

$$x_{i,n}' \in \Omega_n = X_n \setminus \cup_{m \in \Nn} F_n^{-m}(I(F_n))$$

pour $n=0, \cdots , m+n_0$.

Montrons que les points $x_{1,n_0}' , \cdots ,x_{N,n_0}'$ sont $(m, \frac{\epsilon}{8})$-séparés pour $F_{n_0}$.

Soit $1 \leq i,j \leq N$ avec $i \neq j$. Il existe $1 \leq l \leq m-1$ avec

\begin{equation*}
\begin{split}
\epsilon \leq \delta(\sigma^l(\widehat{x_i}),\sigma^l(\widehat{x_j})) &= \sum_{p=0}^{+ \infty} \frac{ dist_p((\sigma^l(\widehat{x_i}))_p , (\sigma^l(\widehat{x_j}))_p)}{2^p diam(X_p)}\\ 
& \leq \sum_{p=0}^{n_0 - 1} \frac{ dist_p((\sigma^l(\widehat{x_i}))_p , (\sigma^l(\widehat{x_j}))_p)}{2^p diam(X_p)} + \frac{\epsilon}{4}\\
& \leq \sum_{p=0}^{n_0 - 1} \frac{ dist_p((\sigma^l(\widehat{x_i}))_p , (\sigma^l(\widehat{x_i}'))_p)}{2^p diam(X_p)} + \sum_{p=0}^{n_0 - 1} \frac{ dist_p((\sigma^l(\widehat{x_i}'))_p , (\sigma^l(\widehat{x_j}'))_p)}{2^p diam(X_p)}\\
& + \sum_{p=0}^{n_0 - 1} \frac{ dist_p((\sigma^l(\widehat{x_j}'))_p , (\sigma^l(\widehat{x_j}))_p)}{2^p diam(X_p)} + \frac{\epsilon}{4}.\\
\end{split}
\end{equation*}

Mais par le calcul fait juste avant le lemme \ref{lemme4} on a 

$$(\sigma^l(\widehat{x_i}))_p =s_{p+1}\circ \cdots \circ s_{p+l}(x_{i,p+l})=s_{p+1} \circ \cdots \circ s_{p+l} \circ \pi_{p+l+1} \circ \cdots \circ \pi_{m+n_0}(x_{i,m+n_0}) $$

pour $p \geq 0$ (et de même pour $\widehat{x_i}'$, $\widehat{x_j}$ et $\widehat{x_j}'$).

En particulier,

\begin{equation*}
\begin{split}
&\sum_{p=0}^{n_0 - 1} \frac{ dist_p((\sigma^l(\widehat{x_i}))_p , (\sigma^l(\widehat{x_i}'))_p)}{2^p diam(X_p)}\\
&= \sum_{p=0}^{n_0 - 1} \frac{ dist_p(s_{p+1} \circ \cdots \circ s_{p+l}(x_{i,p+l}) , s_{p+1} \circ \cdots \circ s_{p+l}(x_{i,p+l}'))}{2^p diam(X_p)}\\
&\leq 
\sum_{p=0}^{n_0 - 1} \frac{\epsilon /8}{ 2^p diam(X_p)} \leq \frac{\epsilon}{4}
\end{split}
\end{equation*}

(pour la dernière inégalité, on utilise que $diam(X_p) \geq diam(X_0)$ grâce à l'hypothèse sur les distances et on peut supposer que le diamètre de $X_0$ vaut au moins $1$).

Il en est de même en remplaçant $i$ par $j$.

Ainsi

$$\epsilon \leq \delta(\sigma^l(\widehat{x_i}),\sigma^l(\widehat{x_j})) \leq \sum_{p=0}^{n_0 - 1} \frac{ dist_p((\sigma^l(\widehat{x_i}'))_p , (\sigma^l(\widehat{x_j}'))_p)}{2^p diam(X_p)} + \frac{3 \epsilon}{4}.$$

En utilisant le lemme \ref{lemme4}, on a

$$(\sigma^l(\widehat{x_i}'))_p =s_{p+1}\circ \cdots \circ s_{p+l}(x_{i,p+l}')=F_p^l \circ \pi_{p+1} \circ \cdots \circ \pi_{p+l}(x_{i,p+l}')=F_p^l(x_{i,p}') $$

pour $p=0, \cdots ,n_0 -1$ et $l=0, \cdots , m-1$ car $x_{i,n}' \in \Omega_n$ pour $n=0, \cdots, m+n_0$, ce qui implique que

$$\frac{ \epsilon}{4} \leq \sum_{p=0}^{n_0 - 1} \frac{ dist_p(F_p^l(x_{i,p}') , F_p^l(x_{j,p}'))}{2^p diam(X_p)}.$$

Maintenant, pour $p=1 , \cdots , n_0$, on a

\begin{equation*}
\begin{split}
dist_p(F_p^l(x_{i,p}') , F_p^l(x_{j,p}')) &\geq dist_{p-1}(\pi_p(F_p^l(x_{i,p}')) , \pi_p(F_p^l(x_{j,p}')))\\
& = dist_{p-1}(F_{p-1}^l(\pi_p(x_{i,p}')) , F_{p-1}^l(\pi_p(x_{j,p}')))\\
&= dist_{p-1}(F_{p-1}^l(x_{i,p-1}') , F_{p-1}^l(x_{j,p-1}'))
\end{split}
\end{equation*}

toujours parce que les $x_{i,n}'$ sont dans $\Omega_n$ pour $n=0, \cdots, m+n_0$ et $i=1, \cdots,N$.

En itérant cette inégalité, on obtient ainsi

$$\frac{ \epsilon}{4} \leq \sum_{p=0}^{n_0 - 1} \frac{ dist_{n_0}(F_{n_0}^l(x_{i,n_0}') , F_{n_0}^l(x_{j,n_0}'))}{2^p diam(X_p)} \leq 2 dist_{n_0}(F_{n_0}^l(x_{i,n_0}') , F_{n_0}^l(x_{j,n_0}')).$$

On a bien montré que les points $x_{i,n_0}'$ sont $(m, \frac{\epsilon}{8})$-séparés pour $F_{n_0}$.

Comme cette propriété est vraie pour tout $m \geq 1$, on a

\begin{equation*}
\begin{split}
&\limsup_{m \rightarrow + \infty} \frac{1}{m} \log \max ( \# G \mbox{, } G  \mbox{ ensemble } (m, \frac{\epsilon}{8}) \mbox{-séparé dans } X_{n_0} \mbox{ pour } F_{n_0} )\\
 &\geq \limsup_{m \rightarrow + \infty} \frac{1}{m} \log \max ( \# G \mbox{, } G  \mbox{ ensemble } (m, \epsilon) \mbox{-séparé dans } X_{\infty} \mbox{ pour } \sigma ) \geq h_{top}(\sigma) - \gamma.
\end{split}
\end{equation*}

Par ailleurs, comme 

$$h_{top}(F_{n_0}) \geq \limsup_{m \rightarrow + \infty} \frac{1}{m} \log \max ( \# G \mbox{, } G  \mbox{ ensemble } (m, \frac{\epsilon}{8}) \mbox{-séparé dans } X_{n_0} \mbox{ pour } F_{n_0} ),$$

on a $h_{top}(F_{n_0}) \geq h_{top}(\sigma) - \gamma$ et ainsi

$$\sup_{n \geq 0} h_{top}(F_n) \geq  h_{top}(\sigma) - \gamma.$$

Cela termine la démonstration du théorème.

\section{\bf Un exemple de calcul de l'entropie $h_{top}(\sigma)$}

Considérons un exemple dû à V. Guedj (voir l'exemple 1.4 dans \cite{Gu}). Il s'agit de l'application méromorphe $f: \Pp^2(\Cc) \longrightarrow \Pp^2(\Cc)$ définie par 

$$f([z:w:t])=[z^2:wt+t^2:t^2].$$

L'ensemble d'indétermination est $I=[0:1:0]$, les degrés dynamiques $d_1$ et $d_2$ sont égaux à $2$ et $h_{top}(f)=0$ (voir \cite{Gu}).

Maintenant, on fait un éclatement de $ \Pp^2(\Cc)$ en $I$. On note $\widehat{ \Pp^2(\Cc)}$ la surface complexe compacte ainsi obtenue et $e_1: \widehat{ \Pp^2(\Cc)} \longrightarrow \Pp^2(\Cc)$ l'éclatement.

$\widehat{ \Pp^2(\Cc)} $ est obtenue en recollant $\Pp^2(\Cc) \setminus \{I \}$ avec 

$$\Gamma=\{ ((z,t), [\alpha:\beta])  \in U \times \Pp^1(\Cc) \mbox{  ,  } z \beta= t \alpha \}$$

via $((z,t), [\alpha:\beta]) \longrightarrow [z:1:t]$, où $U$ est un petit voisinage de $(0,0) \in \Cc^2$ (voir \cite{GH} p.182). Ici on a pris la carte $(w=1)$ dans $\Pp^2(\Cc)$.

L'application $f \circ e_1$ est holomorphe en dehors de $e_1^{-1}(I)$. Montrons qu'elle est méromorphe sur le diviseur exceptionnel et cherchons son point d'indétermination.

Pour cela on écrit tout d'abord $e_1$ en coordonnées: on se place sur la carte $(w=1)$ de $ \Pp^2(\Cc)$ et $(\alpha=1)$ de $\widehat{ \Pp^2(\Cc)} $ et on a $e_1(z, \beta)=(z,z \beta) $ d'où

$$f \circ e_1(z, \beta)=[z^2:z \beta+ (z \beta)^2:(z \beta)^2]=[z:\beta+ z \beta^2:z \beta^2]$$

qui est encore méromorphe en $(0,0)$.

Dans l'autre carte $(\beta=1)$ de $\widehat{ \Pp^2(\Cc)} $ on a $e_1(t, \alpha)=(t \alpha,t) $ d'où

$$f \circ e_1(t, \alpha)=[(t \alpha)^2:t+ t^2:t^2]=[t \alpha^2:1+ t:t]$$

qui est holomorphe (car $t$ est proche de $0$).

L'application $f \circ e_1 :\widehat{ \Pp^2(\Cc)}  \longrightarrow \Pp^2(\Cc)$ est donc méromorphe et son point d'indétermination $\widehat{I}$ est $(0,0)$ dans la carte $(\alpha=1)$.

L'application $e_1$ est biméromorphe. On peut donc considérer $G= e_1^{-1} \circ f \circ e_1$ qui est bien définie en dehors d'un sous-ensemble analytique de $\widehat{ \Pp^2(\Cc)}$. En prenant l'adhérence du graphe de $G$ dans $\widehat{ \Pp^2(\Cc)} \times \widehat{ \Pp^2(\Cc)}$ on obtient ainsi une nouvelle application méromorphe $G: \widehat{ \Pp^2(\Cc)} \longrightarrow \widehat{ \Pp^2(\Cc)}$. Comme $f$ est dominante, $G$ l'est aussi.

Notons $dist_0$ la métrique de Fubini-Study de $\Pp^2(\Cc)$. Comme à la fin du paragraphe \ref{indépendance}, on peut munir $\widehat{ \Pp^2(\Cc)}$ d'une métrique $dist'$ qui définit toujours sa topologie avec $dist'(x,y) \geq dist_0(e_1(x), e_2(y))$.

En notant $F_0=f$, $X_0=\Pp^2(\Cc)$, l'espace $\widehat{ \Pp^2(\Cc)}$ n'est pas le $X_1$ recherché car $f \circ e_1$ est encore méromorphe en $\widehat{I}$.

C'est pourquoi, nous éclatons maintenant $\widehat{ \Pp^2(\Cc)}$ en ce point. Nous obtenons ainsi une nouvelle surface complexe compacte $X_1$ et on note $e_2: X_1 \longrightarrow \widehat{ \Pp^2(\Cc)}$ l'application éclatement.

Montrons que $f \circ e_1 \circ e_2$ est holomorphe. 

$X_1$ est obtenue en recollant $\widehat{ \Pp^2(\Cc)} \setminus \widehat{I}$ avec

$$\Gamma'=\{ ((z,\beta), [u:v])  \in V \times \Pp^1(\Cc) \mbox{  ,  } z v= \beta u \}$$

en utilisant $((z,\beta), [u:v]) \longrightarrow (z , \beta)$ (où $V$ est un petit voisinage de $(0,0) \in \Cc^2$).

L'application $f \circ e_1 \circ e_2$ est holomorphe en dehors de $e_2^{-1}(\widehat{I})$. Pour les autres points écrivons $f \circ e_1 \circ e_2$ en coordonnées.

Dans la carte $(u=1)$, on a $e_2(z,v)=(z, zv)$ d'où

$$f  \circ e_1 \circ e_2(z,v)= f \circ e_1(z, zv)= [z:zv+z(zv)^{2}:z(zv)^{2}]=[1:v+(zv)^{2}:(zv)^{2}]$$

qui est holomorphe et dans la carte $(v=1)$, on a $e_2(\beta,u)=(\beta u, \beta)$ d'où

$$f  \circ e_1 \circ e_2(\beta,u)= f \circ e_1(\beta u, \beta)=[\beta u: \beta + \beta u \beta^2:\beta u \beta^2]=[u:  1+u \beta^2: u \beta^2] $$

qui est aussi holomorphe (quand $u$ est proche de $0$).

En notant $F_0=f$, $X_0=\Pp^2(\Cc)$, $\pi_1=e_1 \circ e_2$ et $s_1=f \circ e_1 \circ e_2$, on a donc obtenu le diagramme

$$
\xymatrix{
X_1   \ar[d]_{\pi_1} \ar[rd]^{s_1} \\
X_{0} \ar@{.>}[r]^{F_{0}} & X_{0}
}
$$

avec $\pi_1$ holomorphe, surjective avec ses fibres génériques réduites à un point et $s_1$ holomorphe et surjective.

L'application $\pi_1$ est biméromorphe. On peut donc considérer $F_1= \pi_1^{-1} \circ F_0 \circ \pi_1$ qui est bien définie en dehors d'un sous-ensemble analytique de $X_1$. En prenant l'adhérence du graphe de $F_1$ dans $X_1 \times X_1$ on obtient ainsi une nouvelle application méromorphe $F_1: X_1 \longrightarrow X_1$. Comme $F_0$ est dominante, $F_1$ l'est aussi. On a bien obtenu le diagramme 

$$
\xymatrix{
X_1  \ar@{.>}[r]^{F_1} \ar[d]_{\pi_1} \ar[rd]^{s_1} & X_1 \ar[d]^{\pi_1}\\
X_{0} \ar@{.>}[r]^{F_{0}} & X_{0}
}
$$

Comme à la fin du paragraphe \ref{indépendance}, on peut munir $X_1$ d'une métrique $dist_1$ qui vérifie 

$$dist_1(x,y) \geq dist'(e_2(x),e_2(y)) \geq dist_0(e_1 \circ e_2(x),e_1 \circ e_2(y))=dist_0(\pi_1(x),\pi_1(y)).$$

Montrons maintenant 

\begin{proposition} 
On a $h_{top}(F_1) = \log 2$.
\end{proposition}

\begin{proof}

Tout d'abord, par T.-C. Dinh et N. Sibony (voir \cite{DS1} et \cite{DS2}), on a d'une part que les degrés dynamiques sont des invariants biméromorphes (en particulier les degrés dynamiques de $F_1$ sont égaux à ceux de $F_0$) et d'autre part que

$$ h_{top}(F_1) \leq \max_{i=0,1,2} \log d_i(F_1)= \log 2.$$

Maintenant, la même démonstration qu'au paragraphe \ref{croissance} implique que $h_{top}(F_1) \geq  h_{top}(G)$. Pour montrer la proposition il suffit donc de voir que $ h_{top}(G)\geq \log 2$.

Dans la carte $(\beta=1)$ de $\widehat{ \Pp^2(\Cc)} $ on a vu que 

$$F_0 \circ e_1(t, \alpha)=[(t \alpha)^2:t+ t^2:t^2]=[t \alpha^2:1+ t:t].$$

On a donc, pour $t$ proche de $0$,

\begin{equation*}
\begin{split}
G(t, \alpha)&= \left(  \left( \frac{t \alpha^2}{1+t}, \frac{t }{1+t} \right), \left[ \frac{t \alpha^2}{1+t} : \frac{t }{1+t} \right] \right)\\
&= \left(  \left( \frac{t \alpha^2}{1+t}, \frac{t}{1+t} \right), [  \alpha^2 : 1 ] \right).
\end{split}
\end{equation*}

Si on considère la carte $(\beta=1)$ au but, cela s'écrit $G(t, \alpha)=\left( \frac{t }{1+t} , \alpha^2 \right)$.

Le diviseur exceptionnel $E$ a pour équation $(t=0)$ dans cette carte de $\widehat{ \Pp^2(\Cc)} $. Si on considère le cercle $| \alpha |=1$ dans $E$, il est invariant par $G$ et la dynamique dessus est $\alpha \longrightarrow \alpha^2$. Cela implique que $h_{top}(G) \geq \log 2$.

\end{proof}

Considérons une suite de diagramme comme dans l'introduction que l'on construit à partir de $X_1$ et $F_1$. On a vu que $h_{top}(\sigma)= \sup_n h_{top}(F_n)$ ce qui implique que $h_{top}(\sigma) \geq h_{top}(F_1)\geq \log 2$. Mais comme les degrés dynamiques des $F_n$ sont égaux à ceux de $F_0$, par \cite{DS1} et \cite{DS2}, on a $\sup_n h_{top}(F_n) \leq \log 2$ c'est-à-dire,

\begin{corollaire}
On a $h_{top}(\sigma)= \log 2$.
\end{corollaire}

\section{\bf{Démonstration du principe variationnel}}
 
Dans ce paragraphe, on suppose que les $\widetilde{f^{-m}(I)}$ sont disjoints (pour $m \in \Nn$). Fixons $n \in \Nn$ et considérons l'application $F_n: X_n \longrightarrow X_n$ comme dans l'introduction. Son ensemble d'indétermination sera encore noté $I(F_n)$. Dans un premier temps nous allons montrer le

\begin{lemme}

Pour $m \in \Nn$, les ensembles $\widetilde{F_n^{-m}(I(F_n))}$ sont disjoints.

\end{lemme}

\begin{proof}

Nous démontrons ce résultat par récurrence sur $n$.

Pour $n=0$, c'est l'hypothèse car $F_0=f$. On suppose maintenant la propriété vraie au rang $n-1$ avec $n \geq 1$. Si les $\widetilde{F_n^{-m}(I(F_n))}$ ne sont pas disjoints, soit

$$x \in \widetilde{F_n^{-m}(I(F_n))} \cap \widetilde{F_n^{-q}(I(F_n))}$$

avec $q > m$. Les entiers $m$ et $q$ sont choisis minimaux, c'est-à-dire que l'on prend le plus petit $m \geq 0$ tel que $\widetilde{F_n^{-m}(I(F_n))} $ rencontre un autre $\widetilde{F_n^{-q}(I(F_n))}$, puis le plus petit $q$ qui vérifie cette propriété (on a donc $q > m$).

Soit $\Gamma_{F_n}$ le graphe de $F_n$ dans $X_n \times X_n$. Par récurrence, on voit que pour $m \geq 1$, $x \in \widetilde{F_n^{-m}(A)}$ est équivalent à l'existence de points $x_0, \cdots,x_m$ avec $(x_i,x_{i+1}) \in \Gamma_{F_n}$ pour $i=0, \cdots, m-1$ (nous appellerons {\bf chaîne} de $F_n$ une telle suite), $x_0=x$ et $x_m \in A$.

Ici, on a $x \in \widetilde{F_n^{-m}(I(F_n))} \cap \widetilde{F_n^{-q}(I(F_n))}$. Il existe donc deux chaînes $x_0, \cdots, x_m$ et $x_0', \cdots,x_q'$ avec $x_0=x_0'=x$, $x_m \in I(F_n)$ et $x_q' \in I(F_n)$.

Comme le $m$ est minimal, les points $x_0, \cdots , x_{m-1}$ ne sont pas dans $I(F_n)$. Par ailleurs, si $y \in X_n \setminus I(F_n)$, il existe un unique $z \in X_n$ avec $(y,z) \in \Gamma_{F_n}$. De là, on en déduit que

$$x_0=x_0' , \cdots, x_m=x_m'.$$

En particulier, dans la chaîne $x_0', \cdots,x_q'$, on a $x_m'$ et $x_q'$ qui sont dans $I(F_n)$.

Remarquons que pour $y,z \in X_n$ avec $(y,z) \in \Gamma_{F_n}$, on a $(\pi_n(y),\pi_n(z)) \in \Gamma_{F_{n-1}}$. En effet, le graphe $\Gamma_{F_n}$ est l'adhérence des points de la forme $\{ (y, F_n(y)) \mbox{  ,  } y \notin I(F_n) \}$. On peut donc trouver une suite $(y_p)$ qui converge vers $y$ avec $y_p \notin I(F_n)$ et $(F_n(y_p))$ qui tend vers $z$ quand $p \rightarrow + \infty$. Comme $\pi_n^{-1}(I(F_{n-1}))$ est un sous-ensemble analytique de $X_n$ (car $\pi_n$ est dominante), quitte à bouger un peu les $y_p$, on peut supposer que $\pi_n(y_p) \notin I(F_{n-1})$. Ainsi, par continuité de $\pi_n$, la suite $(\pi_n(y_p))$ converge vers $\pi_n(y)$ et $F_{n-1} \circ \pi_n(y_p) = \pi_n \circ F_n(y_p)$ vers $\pi_n(z)$ quand $p \rightarrow + \infty$. Cela signifie bien que $(\pi_n(y),\pi_n(z)) \in \Gamma_{F_{n-1}}$.

En utilisant cette remarque, on obtient que $\pi_n(x_0'), \cdots,\pi_n(x_q')$ est une chaîne pour $F_{n-1}$.

Maintenant, montrons le

{\bf Fait:} On a $I(F_n) \subset \pi_n^{-1}(F_{n-1}^{-1}(I(F_{n-1})))$. 

Soit $y \in X_n \setminus \pi_n^{-1}(F_{n-1}^{-1}(I(F_{n-1})))$.  Il y a deux possibilités:

- si $\pi_n(y) \in I(F_{n-1}) $, alors $s_n(y) \notin I(F_{n-1})$ sinon comme $\pi_n(y), s_n(y)$ est une chaîne pour $F_{n-1}$, on aurait $I(F_{n-1}) \cap F_{n-1}^{-1}(I(F_{n-1})) \neq \emptyset$ ce qui contredirait l'hypothèse de récurrence. Ainsi $s_n(y) \notin I(F_{n-1})$ et alors $\pi_n^{-1} \circ s_n$ est holomorphe en $y$ car
$$ \{ x \in X_{n-1} \mbox{  ,  } \mbox{dim}({\pi_n^{-1}(x)}) \geq 1 \} \subset I(F_{n-1}) .$$
Cela implique que $y \notin I(F_n)$.

- si $\pi_n(y) \notin I(F_{n-1}) $, alors $F_{n-1} \circ \pi_n(y)$ est bien défini et ne se trouve pas dans $I(F_{n-1})$. En particulier, $\pi_n^{-1} \circ F_{n-1} \circ \pi_n$ est holomorphe en $y$ et donc $y \notin I(F_n)$.

En utilisant ce fait, on a que $\pi_n(x_0'), \cdots,\pi_n(x_q')$ forme une chaîne pour $F_{n-1}$ avec $\pi_n(x_m') \in F_{n-1}^{-1}(I(F_{n-1}))$ et $\pi_n(x_q') \in F_{n-1}^{-1}(I(F_{n-1}))$. Autrement dit,

$$F_{n-1}^{-1}(I(F_{n-1})) \cap \widetilde{F_{n-1}^{-q+m-1}(I(F_{n-1}))} \neq \emptyset.$$

Cela contredit l'hypothèse de récurrence et termine ainsi la preuve du lemme.

\end{proof}

Passons maintenant à la démonstration du principe variationnel (le théorème \ref{théorème2}).

Soit $\gamma > 0$. Comme $X_{\infty}$ est un espace métrique compact et que $\sigma$ est continue, on peut appliquer le principe variationnel à ce système dynamique, c'est-à-dire qu'il existe une probabilité ergodique $\widehat{\nu}$, invariante par $\sigma$, telle que 

$$h_{\widehat{\nu}}(\sigma) \geq h_{top}(\sigma) - \gamma.$$

Soit 

$$\widehat{\Omega}=\{ \widehat{x}=(x_n) \in X_{\infty} \mbox{  ,  }  x_n \notin \cup_{m \geq 0} \widetilde{F_n^{-m}(I(F_n))} \mbox{ pour tout } n \in \Nn\}.$$

Montrons tout d'abord que $\widehat{\nu}(\widehat{\Omega})=1$.

Si $\widehat{\nu}(\widehat{\Omega}^c) > 0$, il existe $n,m \in \Nn$ avec 

$$\widehat{\nu}( \{\widehat{x}=(x_n) \in X_{\infty} \mbox{  ,  }  x_n \in \widetilde{F_n^{-m}(I(F_n))} \}) > 0.$$

On note $\widehat{B}$ cet ensemble $\{\widehat{x}=(x_n) \in X_{\infty} \mbox{  ,  }  x_n \in \widetilde{F_n^{-m}(I(F_n))} \}$. Par le théorème de récurrence de Poincaré, il existe $l \geq 1$ avec $\sigma^{-l}(\widehat{B}) \cap \widehat{B} \neq \emptyset$. Soit $\widehat{x} \in \sigma^{-l}(\widehat{B}) \cap \widehat{B}$. Si on écrit $\widehat{x}=(x_n)$, on a déjà vu que

$$\sigma^l(\widehat{x})=(\cdots , s_{p+1}\circ \cdots \circ s_{p+l}(x_{p+l}), \cdots , s_1 \circ \cdots \circ s_l(x_{l})).$$

On a donc $x_n \in \widetilde{F_n^{-m}(I(F_n))}$ et $(\sigma^l(\widehat{x}))_n =s_{n+1}\circ \cdots \circ s_{n+l}(x_{n+l}) \in \widetilde{F_n^{-m}(I(F_n))}$.

Mais

\begin{lemme}

$$\pi_{n+1} \circ \cdots \circ \pi_{n+l}(x_{n+l}), \pi_{n+1} \circ \cdots \pi_{n+l-1} \circ s_{n+l}(x_{n+l}), \cdots , s_{n+1}\circ \cdots \circ s_{n+l}(x_{n+l})$$

est une chaîne pour $F_n$.

\end{lemme}

\begin{proof}

Tout d'abord, pour tout $y \in X_{p+1}$ et tout $p \geq 0$, les points $ \pi_{p+1}(y), s_{p+1}(y)$ forment une chaîne pour $F_p$.

En particulier, $\pi_{p+1} \circ s_{p+2} \circ \cdots \circ s_{n+l}(x_{n+l}), s_{p+1} \circ s_{p+2} \circ \cdots \circ s_{n+l}(x_{n+l})$ est une chaîne pour $F_p$ (pour tout $p=n, \cdots , n+l-1$).

Ensuite, en utilisant la remarque faite dans la preuve du lemme précédent, on a

$$\pi_{n+1} \circ \cdots \circ \pi_{p+1} \circ s_{p+2} \circ \cdots \circ s_{n+l}(x_{n+l}), \pi_{n+1} \circ \cdots \pi_p  \circ s_{p+1} \circ s_{p+2} \circ \cdots \circ s_{n+l}(x_{n+l})$$

qui forme une chaîne pour $F_n$ pour tout $p=n, \cdots , n+l-1$. Cela démontre le lemme.

\end{proof}

Comme $\pi_{n+1} \circ \cdots \circ \pi_{n+l}(x_{n+l})=x_n$, on a donc obtenu une chaîne pour $F_n$ qui part de $x_n \in \widetilde{F_n^{-m}(I(F_n))}$ et qui va jusqu'à $s_{n+1}\circ \cdots \circ s_{n+l}(x_{n+l}) \in \widetilde{F_n^{-m}(I(F_n))}$. Cela implique que

$$\widetilde{F_n^{-m}(I(F_n))} \cap \widetilde{F_n^{-m-l}(I(F_n))} \neq \emptyset$$

et on obtient ainsi une contradiction.

Nous avons donc montré que $\widehat{\nu}(\widehat{\Omega})=1$. 

Considérons maintenant $p: X_{\infty} \longrightarrow X_0=X$ qui à $\widehat{x}=(x_n)$ associe $p(\widehat{x})=x_0$.

Tout d'abord, par définition de $\widehat{\Omega}$, on a $p( \widehat{\Omega}) \subset \Omega= X \setminus \cup_{n \geq 0} f^{-n}(I)$. Ensuite, $p$ est injective sur $\widehat{\Omega}$. On a même une propriété un peu plus forte: en effet, soit $p(\widehat{x})=x_0=p(\widehat{x'})$ avec $\widehat{x}=(x_n) \in \widehat{\Omega}$ et $\widehat{x'}=(x_n')$ un point quelconque de $X_{\infty}$. Si $\widehat{x} \neq \widehat{x'}$, on considère $l \geq 1$ le plus entier tel que $x_l \neq x_l'$. Mais $x_{l-1}=\pi_l(x_l)=\pi_l(x_l')$ n'est pas dans $I(F_{l-1})$ car $\widehat{x} \in \widehat{\Omega}$ et par hypothèse

$$ \{ x \in X_{l-1} \mbox{  ,  } \mbox{dim}({\pi_l^{-1}(x)}) \geq 1 \} \subset I(F_{l-1}) .$$

Autrement dit, $\pi_l(x_l)$ et $\pi_l(x_l')$ vivent là où $\pi_l^{-1}$ est holomorphe: on a donc $x_l=x_l'$, ce qui est une contradiction.

Montrons que $p$ est un homéomorphisme de $\widehat{\Omega}$ sur son image. Tout d'abord, par définition de $\delta$, 

$$\delta(\widehat{x},\widehat{x'}) \geq \frac{dist_0(x_0,x_0')}{diam(X_0)}=\frac{dist_0(p(\widehat{x}),p(\widehat{x'}))}{diam(X_0)}$$

donc $p$ est continue (sur tout $X_{\infty}$). Ensuite, soit $x_0 \in p(\widehat{\Omega})$ et $(y_n)$ une suite de $p(\widehat{\Omega})$ qui converge vers $x_0$. Si la suite $(p^{-1}(y_n))$ ne converge pas vers $p^{-1}(x_0)$, il existe $\epsilon > 0$ et une sous-suite $(p^{-1}(y_{\psi(n)}))$ qui reste à distance au moins $\epsilon$ de $p^{-1}(x_0)$ pour la métrique $\delta$. Comme 
$X_{\infty}$ est compact, on peut trouver une sous-suite $(p^{-1}(y_{\varphi(n)}))$ de $(p^{-1}(y_{\psi(n)}))$ qui converge vers $\widehat{z} \in X_{\infty}$. On a déjà que $\delta(\widehat{z},p^{-1}(x_0)) \geq \epsilon$. Ensuite, la continuité de $p$ sur tout $X_{\infty}$ implique que $(p(p^{-1}(y_{\varphi(n)})))$ converge vers $p(\widehat{z})$, autrement dit, $p(\widehat{z})=x_0$. Comme $x_0 \in p(\widehat{\Omega})$ la propriété plus forte que l'injectivité que l'on a montrée donne que $\widehat{z}=p^{-1}(x_0)$. Cela contredit $\delta(\widehat{z},p^{-1}(x_0)) \geq \epsilon$. L'application $p$ est donc un homéomorphisme de $\widehat{\Omega}$ sur $p(\widehat{\Omega})$. Enfin, on a le diagramme commutatif

$$
\xymatrix{
\widehat{\Omega}  \ar[r]^{\sigma} \ar[d]_{p}  & \widehat{\Omega} \ar[d]^{p}\\
p(\widehat{\Omega})  \ar[r]^{f} & p(\widehat{\Omega})
}
$$

En effet, soit $\widehat{x}=(x_n) \in \widehat{\Omega}$, on a $\sigma(\widehat{x})=(s_n(x_n))_{n \geq 1}$, d'où $p(\sigma(\widehat{x}))=s_1(x_1)$. Le point $x_0=\pi_1(x_1)$ n'est pas dans $I(F_0)=I$ car $\widehat{x} \in \widehat{\Omega}$. On a donc 

$$p(\sigma(\widehat{x}))=s_1(x_1)= f \circ \pi_1(x_1)=f(x_0)=f \circ p(\sigma(\widehat{x})).$$

De là, on obtient que la mesure $\nu= p_*(\widehat{\nu})$ est invariante par $f$, vit dans $\Omega$ et son entropie est égale à celle de $\widehat{\nu}$.

En particulier,

$$\sup \{h_{\mu}(f) \mbox{  ,  } \mu \mbox{ ergodique et } \mu(I)=0 \} \geq h_{\nu}(f)=h_{\widehat{\nu}}(\sigma) \geq h_{top}(\sigma) - \gamma.$$

Quand on a une probabilité invariante $\mu$ avec $\mu(I)=0$ alors $\mu(\Omega)=1$ par invariance. Comme on a toujours $\sup \{h_{\mu}(f) \mbox{  ,  } \mu \mbox{ ergodique et } \mu(\Omega)=1 \} \leq h_{top}(f) \leq h_{top}(\sigma)$ (voir \cite{Gu} pour la première inégalité et le théorème \ref{théorème1} pour la seconde), le principe variationnel est démontré.

\bigskip

\bigskip\noindent
Henry De Thélin, Université Paris 13, Sorbonne Paris Cité, LAGA, CNRS (UMR 7539), F-93430, Villetaneuse, France.\\
{\tt dethelin@math.univ-paris13.fr}


\newpage

 

\begin{thebibliography}{00}

\bibitem{Bl}  A. Blanchard, \textit{Sur les variétés analytiques complexes}, Ann. Sci. Ecole Norm. Sup., {\bf 73} (1956), 157-202.

\bibitem{SFJ} S. Boucksom, C. Favre et M. Jonsson,  \textit{Degree growth of meromorphic surface maps}, Duke Math. J., {\bf 141} (2008), 519-538.

\bibitem{Ca} S. Cantat, \textit{Sur les groupes de transformations birationnelles des surfaces}, Ann. of Math., {\bf 174} (2011), 299-340.

\bibitem{Det} H. De Thélin, \textit{Sur les exposants de Lyapounov des applications
  méromorphes}, Invent. Math., {\bf 172} (2008), 89-116.

\bibitem{DS1} T.-C. Dinh et N. Sibony, \textit{Regularization of currents and entropy}, Ann. Sci. Ecole Norm. Sup., {\bf 37} (2004), 959-971.

\bibitem{DS2} T.-C Dinh et N. Sibony, \textit{Une borne supérieure pour l'entropie topologique
  d'une application rationnelle}, Ann. of Math., {\bf 161} (2005),
  1637-1644.

 \bibitem{GH} P. Griffiths et J. Harris,  \textit{Principles of algebraic geometry}, Pure and Applied Mathematics, Wiley-Interscience, 1978.

 \bibitem{Gr} M. Gromov, \textit{On the entropy of holomorphic maps},
  Enseign. Math., {\bf 49} (2003), 217-235.

\bibitem{Gu} V. Guedj, \textit{Entropie topologique des applications méromorphes}, Ergodic Theory Dynam. Systems, {\bf 25} (2005), 1847-1855.

\bibitem{Hi} H. Hironaka, \textit{Desingularization of complex-analytic varieties}, Actes Congrès Intern. Math., Tome 2 (1970), 627-631.

\bibitem{HP} J.H. Hubbard et P. Papadopol, \textit{Newton's method applied to two quadratic equations in $\Cc^2$ viewed as a global dynamical system}, Mem. Amer. Math. Soc., {\bf 191} (2008), n°891.

\bibitem{KH} A. Katok et B. Hasselblatt, \textit{Introduction to the modern theory of dynamical systems}, Encycl. of Math. and its Appl., vol. 54, Cambridge University Press, (1995).

\bibitem{RS} A. Russakovskii et B. Shiffman, \textit{Value distribution for sequences of rational mappings and complex dynamics}, Indiana Univ. Math. J., {\bf 46} (1997), 897-932.

\end{thebibliography}
\end{document}